\UseRawInputEncoding

\documentclass{article}
\usepackage{amsthm,amssymb}
\usepackage{amsmath,graphicx}
 \usepackage[top=2.50cm, bottom=2.50cm]{geometry} 
 \usepackage{helvet}
 \usepackage{amsfonts} 
 \newtheorem{theorem}{Theorem}[section]
 \newtheorem{corollary}{Corollary}[section]
 \newtheorem{definition}{Definition}[section]
 \newtheorem{proposition}{Proposition}[section]
 \newtheorem{lemma}{Lemma}[section]
 \newtheorem{remark}{Remark}[section]
 \newtheorem{example}{Example}[section]
 \usepackage{bm}         
 \usepackage{multirow}
 \usepackage{booktabs}
 \usepackage{algorithm}
 \usepackage{algorithmic}
 \usepackage{caption}
 \usepackage{subfigure}
 \newcommand{\tabincell}[2]{\begin{tabular}{@{}#1@{}}#2\end{tabular}}

 \usepackage{fancyhdr} 
 \usepackage{hyperref} 
 \hypersetup{colorlinks, unicode} 
 \usepackage{multicol}
 \title{\bfseries A skeleton model to enumerate \\standard puzzle sequences}
 \author{Jiaxi Lu\\
 \small Chongqing University\\[-0.8ex]
 \small Chongqing, China\\
 \small\tt jiaxi\_lu@outlook.com\\
 \and
 Yuanzhe Ding\\
 \small Chongqing University\\[-0.8ex]
 \small Chongqing, China\\
 \small\tt dingyuanzhe@yeah.net}
 \date{(Dated: \today)}
 
\begin{document}
 \maketitle
 \begin{abstract}
     Guo-Niu Han [arXiv:2006.14070 [math.CO]] has introduced a new combinatorial object named standard puzzle. We use digraphs to show the relations between numbers in standard puzzles and propose a skeleton model. By this model, we solve the enumeration problem of over fifty thousand standard puzzle sequences. Most of them can be represented by classical numbers, such as Catalan numbers, double factorials, Secant numbers and so on. Also, we prove several identities in standard puzzle sequences.
 \end{abstract}
 
 \noindent{\bfseries Keywords: standard puzzle, integer sequence, secant number, Catalan number, Fibonacci number} 
 \section{Introduction}
     In recent years, there has been tremendous interest in developing new combinatorial model for classical sequences, such as the Catalan numbers \cite{Catalan}, the Euler numbers \cite{sec,D-sec,D-tan,Entringer} and so on.
     In a paper by Guo-Niu Han, he introduced ``a large family of combinatorial objects, called standard puzzles'' \cite{Han}.
     After computing a large number of standard puzzle sequences, he found that only a small part of them already appear in OEIS \cite{OEIS}, and then he gave combinatorial proofs of some of them.
     In 2020, Donald Knuth was concerned with a special case of standard puzzles, the support of which is a set of ``vortices --- that is, it travels a cyclic path from smallest to largest, either clockwise
     or counterclockwise'' \cite{Knuth}. In the end, he found and proved that this sequence is exactly $2n U_n$, in which $U_n$ is a familiar sequence in the OEIS \cite{OEIS}, entry A122647. 
     
     However, in spite of quite a number of contributions dealing with the standard puzzle sequence, there are still a great many standard puzzles that have not been enumerated yet. In this paper, we propose a skeleton model, which leads us to proofs for a series of standard puzzle sequences and additionally we get some equations in enumerating problem of standard puzzle.
     
     One question raised in Han's article is whether there is a standard puzzle sequence 
     in which the secant sequence exists. Using our skeleton model, we define simple pieces and find the standard puzzle sequence of a simple piece \{A, B, D, G, H\} (by the notation in \cite{Han}) is exactly the secant sequence.
     Besides we divide all simple pieces into four classes 1, 2, 3 and 4 according to the relations in each columns. Each class has 20 kinds of simple pieces in it and the standard puzzle sequences of all 80 simple pieces have already been enumerated (Some are trivial or proved by Han \cite{Han}, while the others are shown in this paper.). In addition, we enumerate the sequences of standard puzzles, which satisfy the following conditions respectively. (Only one case in (1) and (2) hasn't been enumerated.)
     
     (1) The union of a 1-simple or 4-simple piece and a set of $B_i$ or $C_i$ forms its support.

     (2) The union of a 1-simple and a 4-simple piece and a set of $B_i$ or $C_i$ forms its support.

     To sum up, we enumerate $(20-1)\cdot2^6\cdot2\cdot2 (in (1))+(20-1)\cdot(20-1)\cdot2^6\cdot2 (in (2))=51072$ standard puzzle sequences by our skeleton model.

     This paper is organized as follows. We devote Section 2 to recall the definition of standard puzzles and define some notation to be used in this paper. In Section 3, we propose the skeleton model and by this model we define simple pieces and enumerate the sequences of all kinds of simple pieces.
     Besides, we enumerate some other standard puzzles by the skeleton model in Section 4.
     Finally, the proofs of some identities in standard puzzle sequences are presented in Section 5.
 \section{Definition and notation}
     Although Han proposed the definition of the standard puzzle combinatorial object in \cite{Han}, we recall it here to make the paper self-contained.\label{Se:2}
 \begin{definition}
     A piece denotes a square that has four numbers, called labels, written on its corners. And a puzzle is a connected arrangement of pieces in the $\mathbb{Z} \times \mathbb{Z} $-plane such that the joining corners of all the pieces have the same labels. A puzzle such that the set of all its labels is simply \{1, 2, \ldots , m\} is called a standard puzzle. In this case m is the number of vertices of the puzzle. 
 \end{definition}
 \begin{remark}
     In this paper, we focus our attention on the standard puzzles of a special shape, which are n standard pieces concatenated in a row. We call these puzzles standard n-puzzles.
 \end{remark}
     By definition, we can see that the four labels of a piece occurring in a standard puzzle are distinct. This means that we can replace the four labels with \{1, 2, 3, 4\} respecting the label ordering which yields a \emph{standard piece}. We represent this operation with a reduction. Apparently, there are $4!=24$ different standard pieces. All of these 24 standard pieces form a set, called $\mathcal{S}\mathcal{P}$. For convenience, we have divided the 24 types of standard pieces into four categories A, B, C, and D according to the order of two numbers in each column. And whthin each category, we have labeled them with numbers 1 to 6, as shown in the Table \ref{1}.
 \begin{definition}
     Given a standard puzzle $\alpha$, a set $\mathcal{P}$ of standard pieces is called a support of the puzzle $\alpha$ if the reduction of each piece occurring in the puzzle $\alpha$ is an element of $\mathcal{P}$ and the minimal support of $\alpha$ is the set of all different pieces of $\alpha$ after reduction.
 \end{definition}
     For convenience, standard puzzles will be represented by two-row matrices. For example, $\begin{bmatrix}3 6 8 7\\ 1 2 4 5\end{bmatrix}$ stands for the following picture. This puzzle contains three pieces, but only two of them are different as standard pieces which are
     $\begin{bmatrix}
         \begin{smallmatrix}
             3 4\\1 2
         \end{smallmatrix}
     \end{bmatrix}$ 
     and 
     $\begin{bmatrix}
        \begin{smallmatrix}
            4 3\\1 2
        \end{smallmatrix}
     \end{bmatrix}$.
     So the minimal support is the set $\bigl\{
     \begin{bmatrix}
        \begin{smallmatrix}
            3 4\\1 2
        \end{smallmatrix}
     \end{bmatrix}, 
     \begin{bmatrix}
        \begin{smallmatrix}
            4 3\\1 2
        \end{smallmatrix}
     \end{bmatrix}\bigr\}$ and each set of pieces containing these two pieces is a support of the puzzle.

     \begin{figure}[h]
         \centering
         \includegraphics[scale=0.4]{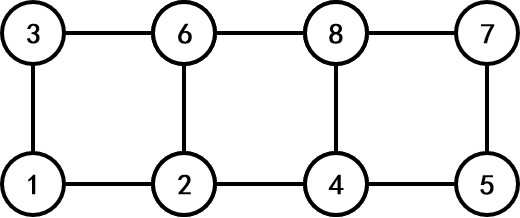}
         \caption{A standard 3-puzzle}
     \end{figure}
 \begin{table}[h]
     \[\begin{aligned} 
        & A_1=A=
                \begin{bmatrix}
                    \begin{smallmatrix}
                        4 3\\1 2
                    \end{smallmatrix}
                \end{bmatrix}&
            & A_2=B=
                \begin{bmatrix}
                    \begin{smallmatrix}
                        3 4\\1 2
                    \end{smallmatrix}
                \end{bmatrix}&
            & A_3=D=
                \begin{bmatrix}
                    \begin{smallmatrix}
                        2 4\\1 3
                    \end{smallmatrix}
                \end{bmatrix}&
            & A_4=H=
                \begin{bmatrix}
                    \begin{smallmatrix}
                        3 4\\2 1
                    \end{smallmatrix}
                \end{bmatrix}&
            & A_5=G=
                \begin{bmatrix}
                    \begin{smallmatrix}
                        4 3\\2 1
                    \end{smallmatrix}
                \end{bmatrix}&
            & A_6=N=
                \begin{bmatrix}
                    \begin{smallmatrix}
                        4 2\\3 1
                    \end{smallmatrix}
                \end{bmatrix}&
        \\
            & B_1=C=
                \begin{bmatrix}
                    \begin{smallmatrix}
                        4 2\\1 3
                    \end{smallmatrix}
                \end{bmatrix}&
            & B_2=E=
                \begin{bmatrix}
                    \begin{smallmatrix}
                        3 2\\1 4
                    \end{smallmatrix}
                \end{bmatrix}&
            & B_3=F=
                \begin{bmatrix}
                    \begin{smallmatrix}
                        2 3\\1 4
                    \end{smallmatrix}
                \end{bmatrix}&
            & B_4=L=
                \begin{bmatrix}
                    \begin{smallmatrix}
                        3 1\\2 4
                    \end{smallmatrix}
                \end{bmatrix}&
            & B_5=J=
                \begin{bmatrix}
                    \begin{smallmatrix}
                        4 1\\2 3
                    \end{smallmatrix}
                \end{bmatrix}&
            & B_6=Q=
                \begin{bmatrix}
                    \begin{smallmatrix}
                        4 1\\3 2
                    \end{smallmatrix}
                \end{bmatrix}&
        \\
            & C_1=X=
                \begin{bmatrix}
                    \begin{smallmatrix}
                        1 3\\4 2
                    \end{smallmatrix}
                \end{bmatrix}&
            & C_2=R=
                \begin{bmatrix}
                    \begin{smallmatrix}
                        1 4\\3 2
                    \end{smallmatrix}
                \end{bmatrix}&
            & C_3=K=
                \begin{bmatrix}
                    \begin{smallmatrix}
                        1 4\\2 3
                    \end{smallmatrix}
                \end{bmatrix}&
            & C_4=P=
                \begin{bmatrix}
                    \begin{smallmatrix}
                        2 4\\3 1
                    \end{smallmatrix}
                \end{bmatrix}&
            & C_5=V=
                \begin{bmatrix}
                    \begin{smallmatrix}
                        2 3\\4 1
                    \end{smallmatrix}
                \end{bmatrix}&
            & C_6=U=
                \begin{bmatrix}
                    \begin{smallmatrix}
                        3 2\\4 1
                    \end{smallmatrix}
                \end{bmatrix}&
        \\
            & D_1=Z=
                \begin{bmatrix}
                    \begin{smallmatrix}
                        1 2\\4 3
                    \end{smallmatrix}
                \end{bmatrix}&
            & D_2=T=
                \begin{bmatrix}
                    \begin{smallmatrix}
                        1 2\\3 4
                    \end{smallmatrix}
                \end{bmatrix}&
            & D_3=M=
                \begin{bmatrix}
                    \begin{smallmatrix}
                        1 3\\2 4
                    \end{smallmatrix}
                \end{bmatrix}&
            & D_4=S=
                \begin{bmatrix}
                    \begin{smallmatrix}
                        2 1\\3 4
                    \end{smallmatrix}
                \end{bmatrix}&
            & D_5=Y=
                \begin{bmatrix}
                    \begin{smallmatrix}
                        2 1\\4 3
                    \end{smallmatrix}
                \end{bmatrix}&
            & D_6=W=
                \begin{bmatrix}
                    \begin{smallmatrix}
                        3 1\\4 2
                    \end{smallmatrix}
                \end{bmatrix}&
        \end{aligned}\]
     \caption{The twenty-four pieces and its codes}
     \label{1}
\end{table}
    \begin{remark}
         In this table, the three columns of each expression represent the notation used in this paper, the notation used in Guo-Niu Han's paper and the matrix representation of standard pieces respectively. Additionally, notice that all the standard pieces $B_i$ and $C_i$ $(1\leq i\leq6)$ have a different relations in the left and right columns so we call $B_i$ and $C_i$ a 1-converter and a 2-converter respectively. 
    \end{remark}

     Then, let us define the standard puzzle sequence.

 \begin{definition}
      Let $\mathcal{P}$ be a subset of $\mathcal{S}\mathcal{P}$. For $n\in \mathbb{N}^+$, we denote by $S_n(\mathcal{P})$ the set of the standard n-puzzles, each of which has $\mathcal{P}$ as its support
      , and let $s_n(\mathcal{P}):=\bigl|S_n(\mathcal{P})\bigr|$ be the cardinality of $S_n(\mathcal{P})$.
     In addition, we call $\{s_n(\mathcal{P})\}_n$ the standard puzzle sequence corresponding to $\mathcal{P}$.
 \end{definition}
 \begin{example}
     Let $\mathcal{P}=\{A_2, A_3\}=\bigl\{
         \begin{bmatrix}3&4 \\ 1 & 2\end{bmatrix},
         \begin{bmatrix}2&4 \\ 1 & 3\end{bmatrix} 
         \bigr\}$,Then,
     \[
         \begin{aligned}
             \{S_1(\mathcal{P})\} =\bigl\{
                 &\begin{bmatrix}3 4 \\ 12\end{bmatrix},
                 \begin{bmatrix}2 4 \\ 1 3\end{bmatrix}
                 \bigr\}\\
             \{S_2(\mathcal{P})\} =\bigl\{
                 &\begin{bmatrix}4 5 6\\ 1 2 3\end{bmatrix},
                 \begin{bmatrix}3 5 6\\ 1 2 4\end{bmatrix},
                 \begin{bmatrix}3 4 6\\ 1 2 5\end{bmatrix},
                 \begin{bmatrix}2 5 6\\ 1 3 4\end{bmatrix},
                 \begin{bmatrix}2 4 6\\ 1 3 5\end{bmatrix}
                 \bigr\}\\
             \{S_3(\mathcal{P})\} =\bigl\{
                 &\begin{bmatrix}5 6 7 8\\ 1 2 3 4\end{bmatrix},
                 \begin{bmatrix}4 6 7 8\\ 1 2 3 5\end{bmatrix},
                 \begin{bmatrix}4 5 7 8\\ 1 2 3 6\end{bmatrix},
                 \begin{bmatrix}4 5 6 8\\ 1 2 3 7\end{bmatrix},
                 \begin{bmatrix}3 6 7 8\\ 1 2 4 5\end{bmatrix},\\
                 &\begin{bmatrix}3 5 7 8\\ 1 2 4 6\end{bmatrix},
                 \begin{bmatrix}3 5 6 8\\ 1 2 4 7\end{bmatrix},
                 \begin{bmatrix}3 4 7 8\\ 1 2 5 6\end{bmatrix},
                 \begin{bmatrix}3 4 6 8\\ 1 2 5 7\end{bmatrix},
                 \begin{bmatrix}2 6 7 8\\ 1 3 4 5\end{bmatrix},\\&
                 \begin{bmatrix}2 5 7 8\\ 1 3 4 6\end{bmatrix},
                 \begin{bmatrix}2 5 6 8\\ 1 3 4 7\end{bmatrix},
                 \begin{bmatrix}2 4 7 8\\ 1 3 5 6\end{bmatrix},
                 \begin{bmatrix}2 4 6 8\\ 1 3 5 7\end{bmatrix}
                 \bigr\}
         \end{aligned}
     \]
 \end{example}

     Indeed, the standard puzzle sequence of $\{A_2, A_3\}$ is shown in \cite{Han} to be the sequence of Catalan numbers:
     \[\{2, 5, 14, 42, 132, 429, 1430, 4862, 16796, 58786, 208012,\ldots\}\]

     Moreover, Guo-Niu Han defined three basic transformations and now we rephrase them here using our notation. The original statement of these transformations can be found in \cite{Han}, which are as follows.
     
     (T1) Exchanging left column and right column in every piece;

     (T2) Exchanging top row and bottom row in every piece;

     (T3) Replacing each label a by (5 − a) in every piece.

 \begin{definition}
     (1) We denote by $F_1$ the map: $\mathcal{S}\mathcal{P}\mapsto \mathcal{S}\mathcal{P}$, which sends the element $A_i, B_i, C_i$ and $D_i$ to $A_{i+3}, C_{i+3}, B_{i+3}$ and $D_{i+3}$ respectively, for i $\in$ \{1, 2, 3, 4, 5, 6\}, and we make the convention that $X_j$ is equal to $X_{k}$ iff $j \equiv k\pmod 6 $ $(X\in \{A, B, C, D\})$.

     (2) We denote by $F_2$ the map: $\mathcal{S}\mathcal{P}\mapsto \mathcal{S}\mathcal{P}$, which sends the element $A_i, B_i, C_i$ and $D_i$ to $D_i, C_i, B_i$ and $A_i$ respectively, for i $\in$ \{1, 2, 3, 4, 5, 6\}.
    
     (3) We denote by $F_3$ the map: $\mathcal{S}\mathcal{P}\mapsto \mathcal{S}\mathcal{P}$, which sends the element $X_i$ to $Y_j$, for $X\in\{A, B, C, D\} ,\ i \in \{1, 2, 3, 4, 5, 6\}$. Y and j are defined as follows.
         
     \begin{equation*}
         Y=\left\{
             \begin{aligned}
                 &A, &X=D\\
                 &B, &X=C\\
                 &C, &X=B\\
                 &D, &X=A
             \end{aligned}
             \right. 
             \qquad\qquad
         j=\left\{
             \begin{aligned}
                 &1, &i=1\\
                 &2, &i=5\\
                 &3, &i=6\\
                 &4, &i=4\\
                 &5, &i=2\\
                 &6, &i=3
             \end{aligned}
             \right.
     \end{equation*}
     \label{Definition}
 \end{definition}
    
     


    
     This definition naturally gives rise to the following proposition.

     \begin{proposition}
         For a map $F_i ( i\in\{1, 2, 3\} )$, let $\mathcal{P}$ be a subset of $\mathcal{S}\mathcal{P}$, and we have,\label{pr}
         \begin{equation}
             s_n(\mathcal{P}) = s_n(F_i (\mathcal{P}))
         \end{equation}
     \end{proposition}
     \begin{proof}
         By the definition of Guo-Niu Han's transformations, we can easily see that these operations are all bijections between $S_n(\mathcal{P})$ to $S_n(F_i (\mathcal{P}))$. Hence the cardinality $s_n(\mathcal{P})$ is equal to $s_n(F_i(\mathcal{P}))$.
     \end{proof}


     
     


    
 \section{Simple puzzles defined by the skeleton model}
     \subsection{The definition of a skeleton model and a simple puzzle}
     Notice that a standard puzzle can be seen as a graph. And if we use the directions on a graph to represent the relations between numbers in a standard puzzle, we can get a directed graph. Then the idea of the skeleton model is born. In fact this idea is inspired by a proof in \cite{Han}. Here are some definitions of a skeleton and a simple piece.
     \begin{definition}
         Given a simple directed graph with four vertices a, b, c and d, we fix the four vertices on a square as shown in figure \ref{fig:square}. We call this graph a basic skeleton(denoted by B(a, b, c, d)), if it satisfies the following conditions.

         (1) It doesn't have a directed circle.

         (2) If there are more than one directed path between two vertices, then they have the same length.
     \end{definition}
     \begin{figure}[h]
         \centering
         \includegraphics[scale=0.4]{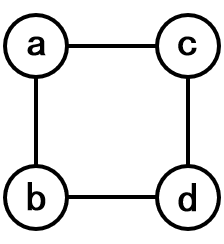}
         \caption{A basic skeleton}
         \label{fig:square}
     \end{figure}
     \begin{remark}
         By definition, we can see that each basic skeleton $B(a, b, c, d)$ is the Hasse diagram of a certain poset. We call this poset the poset of $B$. The definition of the Hasse diagram can be seen in \cite{Stanley}.
     \end{remark}
     \begin{definition}
         A basic skeleton is called a i-basic skeleton
         , if it meets the condition $(i)$, for $1\leq i \leq 4$.
         
         (1)There is a directed path from b to a and from d to c respectively.

         (2)There is a directed path from b to a and from c to d respectively.

         (3)There is a directed path from a to b and from d to c respectively.

         (4)There is a directed path from a to b and from c to d respectively.
     \end{definition}
     \begin{definition}
         Given an i-basic skeleton B(a, b, c, d),
         a i-simple piece generated by $B$ is a subset of $\mathcal{S}\mathcal{P}$(denoted by $\mathcal{P}(B)$). A standard piece is in $\mathcal{P}(B)$ if and only if its poset is a linear extension of the poset of $B$. Then a puzzle which has $\mathcal{P}(B)$ as its support is a simple puzzle generated by $B$.
     \end{definition}
     For example,
     \begin{example}
         Let B be a graph as shown in figure \ref{fig:B}. Then a simple piece generated by B is $\{A_1, A_2, A_3, A_4, A_5\}$ and each standard puzzle of $S(\mathcal{P}(B))$ is a simple puzzle generated by B.\label{ex:secant}
     \end{example}
     \begin{figure}[h]
         \centering
         \includegraphics[scale=0.4]{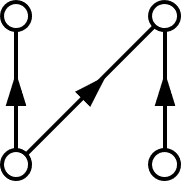}
         \caption{The basic skeleton of $\{A_1, A_2, A_3, A_4, A_5\}$}
         \label{fig:B}
     \end{figure}
     \begin{remark}
         Similar to the definition of the basic skeleton, we can draw a Hasse diagram to 
         descibes the relations between numbers in certain standard n-puzzle. If there is a directed path from x to y it means y > x. We call this graph a skeleton. For example, in example \ref{ex:secant} the skeleton of $\{A_1, A_2, A_3, A_4, A_5\}$ is shown in figure \ref{fig:secant1}.
         \begin{figure}[h]
             \centering
             \includegraphics[scale=0.4]{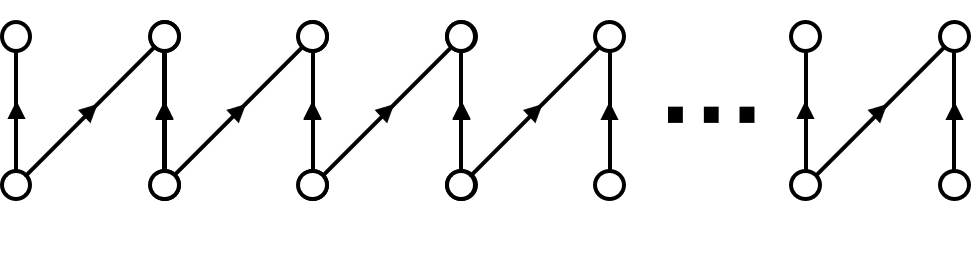}
             \caption{The skeleton of $\{A_1, A_2, A_3, A_4, A_5\}$}
             \label{fig:secant1}
         \end{figure} 

         In addition, both a 2-simple piece and a 3-simple piece have a different relations in the left and right columns. This means that we can not concatenate any one of them in a line. So for a 2-simple or 3-simple piece $\mathcal{P}(B)$, $s_n(\mathcal{P}(B))=0$ for $n\geq 2$. And by the proposition \ref{pr}, for any 4-simple piece $\mathcal{P}(B)$, we can find a 1-simple piece $F_2(\mathcal{P}(B))$such that 
         \begin{equation*}
             s_n(\mathcal{P}(B)) = s_n(F_2 (\mathcal{P}(B)))
         \end{equation*}
         Therefore, in this section we will concentrate on 1-simple puzzles to get the properties of all four kinds of simple puzzles.
    
    
    
    
    
    
 \end{remark}
    
         By definition, we can learn that there are twenty\footnote{By definition of basic skeletons, we just have to think about the four edges in a 4-vertices digraph except the edges ab and cd(seeing figure \ref{fig:square}). Then we count the number of 1-simple pieces in order of the number of directed edges in the skeletons and it is $1+9+8+2=20$. }kinds of 1-simple pieces in total.      In the following subsections, we will prove the sequences of all twenty simple pieces as shown in the appendix \ref{table}. (Some proofs have been given by Han in \cite{Han} or trivial, and we list them at the last column as a remark.)
   

     \subsection{Secant numbers}
         In Guo-Niu Han's paper \cite{Han}, he found the tangent number sequence occurs in the standard puzzle sequence of $\{A_1, A_5, B_1, B_2\}$ and ``developed a computer programme to generate the puzzle sequences up to order |$\mathcal{P}$| = 4" \cite{Han}, finding no secant number sequence. Then he asked a question about whether and where the secant numbers exist.

         However, when we are concerned with the standard puzzle sequence of the simple piece in example \ref{ex:secant}, we were surprised to find that this sequence is exactly a secant sequence.

 
         That is,
         \begin{theorem} Let $\mathcal{P}=\{A_1, A_2, A_3, A_4, A_5\}$, then we have,
             \begin{equation}
                 s_n(\mathcal{P})=S(n+1)
             \end{equation}
             In which S(n) is the n-th secant number in the OEIS, entry A000364.\label{thm:Secant}
         \end{theorem}
            
            
   
            
         \begin{proof}

             Firstly, for the 1-basic skeleton B(a, b, c, d) as shown in figure \ref{fig:B}. We check the i-simple piece generated by B and find it is $\{A_1, A_2, A_3, A_4, A_5\}$ exactly.

             
             Additionally, in 1879 D\'{e}sir\'{e} Andr\'{e} \cite{sec} proved that 
             the secant numbers is the number of down-up permutations of $[2n]$. Therefore, the aim is to find a one-to-one correspondence between this puzzle sequence and a down-up permutation of $[2n+2]$.

             On the one hand, we read the numbers of the puzzle from the top left to the bottom right, through south $\rightarrow$ northeast $\rightarrow$ south $\rightarrow$ northeast $\ldots$ Therefore, by definition we get a permutation $\mathcal{P}$ of length $2n+2$, which satisfies $a(i)>a(i+1)$, where $i = 2k+1$ , and $a(i)<a(i+1)$, where $i = 2k$ ($k\in\{0, 1, 2,\ldots, n\}$, a(i) denotes the i-th number in $\mathcal{P}$, and a(0) is defined by 0). So, it's a down-up permutation.
         
             On the other hand, given a down-up permutation, we can place the permutation at the vertices of the puzzle in the way shown above. Obviously, this puzzle is a simple puzzle of $\{A_1, A_2, A_3, A_4, A_5\}$.

             Hence, the proof is completed.
         \end{proof}

     \subsection{Double factorial of odd and even numbers}
         In this subsection we are concerned with 1-basic skeletons as shown in figure \ref{BK:AB} and \ref{BK:ABD}. After checking the simple pieces generated by them, we find that they are $\{A_1, A_2, A_3\}$ and $\{A_1, A_2\}$ respectively.\label{Se:3.3}
         \begin{figure}[h]
             \centering
             \subfigure[the basic skeleton of $\{A_1, A_2, A_3\}$]
             {
                 \begin{minipage}{5cm}
                 \centering          
                 \includegraphics[scale=0.4]{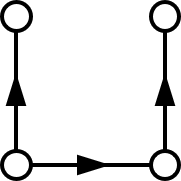}
                 \label{BK:AB}
                 \end{minipage}
             } 
             \subfigure[the basic skeleton of $\{A_1, A_2\}$] 
             {
                 \begin{minipage}{5cm}
                 \centering      
                 \includegraphics[scale=0.4]{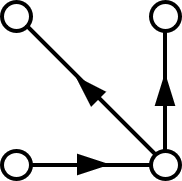}
                 \label{BK:ABD}        
                 \end{minipage}
             }   
             \caption{The basic skeletons}
             \end{figure}
             
             Then we have
     \begin{theorem}

             Let $\mathcal{P}=\{A_1, A_2, A_3\}$ and $\mathcal{Q}=\{A_1, A_2\}$
             \begin{equation}
                s_n(\mathcal{P})=(2n+1)!! \label{a}
             \end{equation}
             \begin{equation}
                 s_n(\mathcal{Q})=(2n)!! \label{b}
             \end{equation}\label{theorem}
     \end{theorem}
         In this paper, we present two proofs of this theorem, one of which is provided by our academic advisor. We are very grateful to Shishuo Fu for this excellent proof, which are as follows. \label{AB}
         \subsubsection{The first method}
             \begin{proof}
                 To begin with, we observed that each piece of $S_n(\mathcal{P})$ and $S_n(\mathcal{Q})$ meets the minimum number in its lower left corner. On this basis, the proof of this theorem is constructed.

                 For equation \ref{a}, we put the smallest number one in $[2n+2]$ on the lower left of the puzzle, and then at the top of the number one, we have $2n+1$ choices from 2 to $2n+2$. After selecting these two numbers in the leftmost column, the second column performs the same action: place the smallest of the remaining numbers at the bottom of the second column and then select one of the remaining numbers and place it on the top of the column, with $(2n-1)$ options. A clear induction gives $(2n+1)(2n-1)\cdot\ldots\cdot 3\cdot 1=(2n+1)!!$ permutations that meet the above conditions.

                 Now, all we need to do is prove that the support of puzzles which meet the above properties is accurately $\{A_1, A_2, A_3\}$. Check all the standard pieces, and pieces satisfying the above conditions are only $A_1, A_2$ and $A_3$.
                 
                 The equation \ref{a} has been shown and we can prove the equation \ref{b} in a similar way. Notice that each piece of $S_n(\mathcal{Q})$ satisfies that the upper-left number is larger than the lower-right number, in addition to the above conditions. Therefore, after setting the number one in the lower left corner, the number above the number one has only $2n$ conditions from 3 to $2n+2$. Then the second column list $(2n-2)$ possibilities. By recursion, the number of the element in $S_n(\mathcal{Q})$ is $(2n)!!$, which completes the proof of the Theorem \ref{theorem}.
             \end{proof}
         \subsubsection{The second method}
             In this subsection, we provide an alternative method for proving the Theorem \ref{theorem}.
             \begin{proof}
                 Proceed by induction on length of $S_n(\mathcal{P})$. If $n = 1$, then this puzzle is $A_1$ or $A_2$ or $A_3$. So, 
                 $s_1(\mathcal{P})=3=(2\cdot 1+1)!!$. 
                 This establishes the base case of the induction.
                 Suppose the claim holds for puzzles of length less than $k$. Then we construct a 1-to-(2n+1) map from $S_{n-1}(\mathcal{P})$ to $S_n(\mathcal{P})$.

                 For any puzzle of $S_{n-1}(\mathcal{P})$, there are $n$ numbers on the top line. We choose one of them and there are $n$ different options. Then let's put it on the second line such that it is bigger than the number to its left, and smaller than the number to its right. In other words, rearrange the columns according to the size of the rows below. After the operation, we can see two empty seats in the first row. One of them is the original position of the number taken out, and the other is the empty seat above its new position. After this, we put the two number $2n+1$ and $2n+2$ in these two space in two ways. Additionally, we can put the number $2n+1$ and $2n+2$ to the bottom right and top right corners of this puzzle respectively and then we can get another puzzle in $S_n(\mathcal{P})$. Besides, this construction is easily seen to be invertible and then the proof is completed.


             \end{proof}

                 To prove the equation \ref{b}, we need to use the skeleton model in the next section, so the proof of it is postponed until Section \ref{section: 4.1}.
     \subsection{The standard puzzle sequence of $\{A_1, A_2, A_4, A_5\}$}
     
     \begin{theorem}Let $\mathcal{P}=\{A_1, A_2, A_4, A_5\}$. Then we have, 
         \begin{equation}
            s_n(\mathcal{P})=L(n+1)
         \end{equation}
         where L(n) is number of lattice paths from $\{2\}^n$ to $\{0\}^n $ using steps that decrement one component by 1 such that for each point $(p_1,p_2, \ldots ,p_n)$ we have $\bigl|p_{i}-p_{i+1}\bigr| \leqslant 1$, which is in OEIS, entry A227656.
     \end{theorem}
     \begin{proof}
         First of all, we draw the skeleton of $\{A_1, A_2, A_4, A_5\}$.

         \begin{figure}[h]
            \centering
            \includegraphics[scale=0.4]{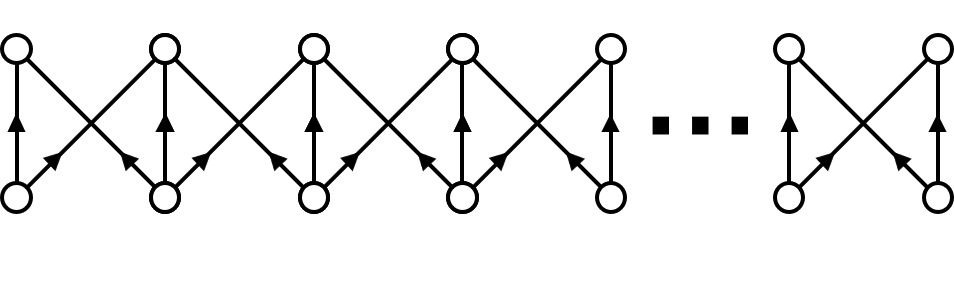}
            \caption{The skeleton of $\{A_1, A_2, A_4, A_5\}$}
            \label{fig:ABGH}
         \end{figure}

         The trick of the proof is to find a bijection between any element of $S_n(\mathcal{P})$ and the path meeting the above conditions of length $n+1$.

         Then, we construct a combinatorial proof of bijection. Check the positions of the numbers in the standard puzzle, from 1 to $2n+2$ one by one. If the number is in column i, we will change $(p_1, p_2, \dots , p_i, \dots, p_{n+1})$ into $(p_1, p_2, \dots , p_i-1, \dots, p_{n+1})$. And the initial point is $\{2\}^n$. This construction is easily seen to be invertible. Then we only need to show this kind of puzzle is exactly $S_n(\mathcal{P})$.

         Notice that the j-th step of the path is equal to a n-set which count the number greater than j in each column of an element of $S_n(\mathcal{P})$. Then that each point $(p_1,p_2, \dots ,p_n)$ satisfies $\bigl|p_{i}-p_{i+1}\bigr| \leqslant 1$ is the same thing as it that the puzzle satisfies that each number on the top row of the puzzle is greater than the numbers below it and below it's adjacent numbers. Therefore, it is equivalent to it that $\{A_1, A_2, A_4, A_5\}$ is a support of the puzzle by the figure \ref{fig:ABGH}
     \end{proof}
     \subsection{Fibonacci sequence}
     In this subsection, we enumerate the standard puzzle sequence of $\{A_1, B_1, C_1\}$ which is a Fibonacci sequence. Although this puzzle is not one of 1-simple puzzles, we list it here for theorem \ref{theorem:P=AP} to be complete.

     Before we state and prove the theorem, we will need a lemma.
         \begin{lemma}
             The number of subsets of $S(n)=\{1, 2, \dots, n\}$ that contain no consecutive integers is equal to F(n+2), where F(n) is the n-th Fibonacci number for F(0)=0, F(1)=1.\label{lemma1}
         \end{lemma} 
         \begin{proof}
             We denote by $s(n)$ the number in this lemma. And supposing the length of $S(n)$ is 1 and 2, we count the number of its subset satisfying the above properties, which is accurately equal to F(3) and F(4). Then we only need to show that $s(n+2) = s(n+1) + s(n)$, for all positive integers n.

             We enumerate the number of subsets of $S(n+2)$ meeting the conditions. Discuss it in two ways whether the number $n+2$ is in the subset or not. Suppose that the number $n+2$ is not in the subset, and then this subset is made up of $\{1, 2, \dots, n+1\}$, that is, it is a subset of $S(n+1)$. 
        
             Then we suppose the number $n+2$ is in the subset. According to the definition of the subset, the number $n+1$ is not in this subset, and other numbers are all made up of $\{1, 2, \dots, n\}$. This means the set is a union of a subset of S(n) and $\{n+2\}$.

             Combining the above two points, the proof of Lemma \ref{lemma1} is completed.
         \end{proof}
         \begin{theorem}
             Let $\mathcal{P}=\{A_1, B_1, C_1\}$. Then the standard puzzle sequence of $\mathcal{P}$ is equal to the Fibonacci sequence.
             
             That is,
             \begin{equation}
                 s_n(\mathcal{P})=F(n+3)
             \end{equation}
             where F(n) is the n-th Fibonacci number for F(0)=0, F(1)=1.
             \label{fibonacci}
         \end{theorem}
         \begin{proof}


             We think of a mapping similar to the mapping in the Theorem 3 of \cite{Han}, called the flip map. Let S be a subset of $\{1, 2, . . . , n+1\}$. The flip map $\phi_S : \alpha \mapsto \beta $ is a transformation which maps a puzzle
             $\alpha = \bigl[\begin{smallmatrix}x_1x_2\dots x_{n+1}\\
             y_1y_2\dots y_{n+1}\end{smallmatrix}\bigr]$ onto $
             \beta = \bigl[\begin{smallmatrix}a_1a_2\dots a_{n+1}\\
             b_1b_2\dots b_{n+1}\end{smallmatrix}\bigr] $
             such that $a_i = x_i$, $b_i = y_i$ for $i \notin S$ and
             $a_i = y_i$, $b_i = x_i$ for $i \in S$. 

             Notice that a map $\phi_S $ can convert a piece $A_1$ into $B_1$ or $C_1$ or $D_1$, for any S. However, the support of the image of a puzzle of $S_n(A_1)$ under the maping $\phi_S $ is $S_n(\mathcal{P})$ if and only if there are no adjacent elements in S. This is because if suppose S has adjacent elements $i$ and $i+1$, then the reduction of the piece $\bigl[\begin{smallmatrix}
                 x_ix_{i+1}\\
                 y_iy_{i+1}
             \end{smallmatrix}\bigr]$
             is exactly $D_1$ which is not in $\{A_1, B_1, C_1\}$. And vice versa.
             
             So, the number $s_n(\mathcal{P})$ is equivalent to the number of subsets of $\{1, 2, . . . , n+1\}$ such that there is no adjacent number in it. Notice that the number $s_1(A_1)$ is exactly one, and according to the Lemma \ref{lemma1}, this theorem has been proved.
         \end{proof}
 \section{Standard puzzles enumerated by the skeleton model}
     \subsection{Standard puzzles relating to double factorials}
         In this subsection, we dicuss about a kind of special standard puzzles such that $\{A_1, A_2, A_3\}$ or $\{A_1, A_2\}$ uniting 1-converters (resp. 2-converters) forms its support. Notice that any two 1-converters (resp. 2-converters) can't appear at the same time in this kind of puzzles. We can simply enumerate the standard puzzle sequence of the union of $\{A_1, A_2, A_3\}$ or $\{A_1, A_2\}$ and a 1-converter (resp. 2-converter). And we can sum them and subtract out the double counting to get all this kind of puzzles.

         Before enumerating these kinds of puzzles, we need some theorems as follows. \label{section: 4.1}
         \begin{definition}
             Let $\mathcal{P}$ be a set of standard pieces and $S_n(\mathcal{P}_x)$(resp. $S_n(\mathcal{P}^x)$) denotes the set of all standard n-puzzles, each of which has $\mathcal{P}$ as its support and the number $x$ is in the bottom right corner (resp. top right corner). And let $s_n(\mathcal{P}_x):=\bigl|S_n(\mathcal{P}_x)\bigr|$ (resp. $s_n(\mathcal{P}^x):=\bigl|S_n(\mathcal{P}^x)\bigr|$) be the cardinality of $S_n(\mathcal{P}_x)$ (resp. $S_n(\mathcal{P}^x)$).  
         \end{definition} 

         \begin{theorem}
             Let $\mathcal{P}$ be $\{A_1, A_2, A_3\}$, then we have,
             \begin{equation}
                s_n(\mathcal{P}_{2n-k+2})=T(n,k) \label{A123}
             \end{equation}
             where T(n,k) is the sum of the weights of all vertices labeled k at depth n in the Catalan tree $(1 \leqslant k \leqslant n+1, n \geqslant 0)$, in OEIS \cite{OEIS} entry A102625.\label{A123_}
         \end{theorem}
         \begin{proof}
             By definition of $T(n,k)$, let $T(m, 0)=0$ for any positive integer $m$, and then we can know that $T(n,k)$ satisfies the following equation.
             \begin{equation}
                 T(n,k)=k\sum_{i=k-1}^nT(n-1,i)
             \end{equation}
             Then, let a standard puzzle of $S_n(\mathcal{P})$ with $2n+2-k$ and $j$ in the bottom and top right hand corner respectively, for $2n+3-k \leqslant j \leqslant 2n+2$. Then we get rid of the rightmost column of this puzzle, and replace all the numbers by $\{1, 2, \dots , 2n\}$ respecting the label ordering. Thus, we get a standard puzzle in $S_{n-1}(\mathcal{P}_{2n-i})$ for $k-1 \leqslant i \leqslant n$ by observing the skeleton of $\mathcal{P}$. Besides, each element of $S_{n-1}(\mathcal{P}_{2n-i})$ can be produced by $k$ different elements of $S_n(\mathcal{P}_{2n-k+2})$ for the number $j$ has $k$ choices. 
             
             That is
             \begin{equation}
                 s_n(\mathcal{P}_{2n-k+2})=
                 k\sum_{i=k-1}^ns_{n-1}(\mathcal{P}_{2n-i})
             \end{equation}
             Then the recurrence relation has been shown. So we only need to check the equation \ref{A123} for $n=1$. Because of $s_1(\mathcal{P}_{3})=1=T(1,1)$ and $s_1(\mathcal{P}_{2})=2=T(1,2)$, the proof of Theorem \ref{A123_} is completed.
         \end{proof}
         Then, we check it in OEIS \cite{OEIS}
         , and find the formula of it.
         \begin{equation}
             T(n,k) = \frac{k(2n-k+1)!}{(n-k+1)!\cdot 2^{n-k+1}} 
         \end{equation} 
         At this point, we can give second proof of Theorem \ref{theorem}.
        \paragraph{Proof of Theorem \ref{theorem}}
         \begin{proof}
             Let $\mathcal{P}$ and $\mathcal{Q}$ denote $\{A_1, A_2\}$ and $\{A_1, A_2, A_3\}$ respectively and now we draw the skeletons of $S_n(\mathcal{P})$ and $S_n(\mathcal{Q})$.

             \begin{figure}[h]
                 \centering
                 \includegraphics[scale=0.3]{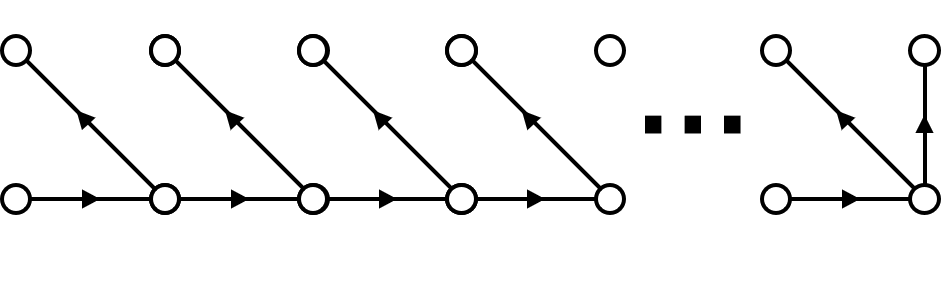}\\
                 \includegraphics[scale=0.3]{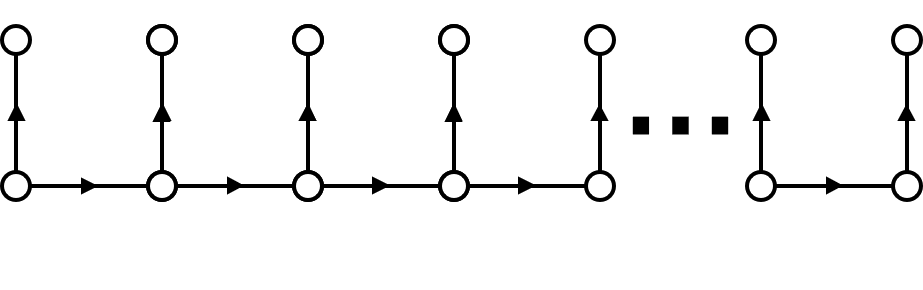}
                \caption{The skeletons of $S_n(\mathcal{P})$ and $S_n(\mathcal{Q})$}
                \label{fig:factorial}
            \end{figure}
             By contrast, we can find the skeletons of $S_n(\mathcal{P})$ and $S_n(\mathcal{Q})$ are very similar. Let us do a simple operation on $S_n(\mathcal{P})$: shift the top row of numbers to the right by one space. Then, we will find it incredible that $S_n(\mathcal{P})$ and $S_{n-1}(\mathcal{Q})$ are the same in the middle, except that $S_n(\mathcal{P})$ has two more leftmost and rightmost digits than $S_{n-1}(\mathcal{Q})$. The leftmost number of $S_n(\mathcal{P})$ is accurately number one, while the rightmost number is uncertain. Given a puzzle of $S_{n-1}(\mathcal{Q})$ we add one to each label of it, and put the number one at the lower left corner in the puzzle. And then, we consider the number at the lower right corner after changing denoted by $2n+1-k$ for $k \in\{1, 2, \dots, n\}$. Then, we choose a number $j \in\{2n+2-k, 2n+3-k, \dots, 2n+2\}$. We can know that there are $k+1$ kinds of way to choose the number for given $k$. Then, we put this number on the top right corner of the new puzzle and add one to each number no less than the number $j$ after the first change. Then shift the top row of numbers to the left by one space. A puzzle $S_n(\mathcal{P})$ has been constructed. Besides, this construction is easily seen to be invertible.

             Thus, we get the number $s_n(\mathcal{P})$, that is,
             \begin{equation}
                 s_n(\mathcal{P})=\sum_{k=1}^n(k+1)s_{n-1}(\mathcal{P}_{2n-k})\label{AB2}
             \end{equation}
             We substitute the equation \ref{A123} into the equation \ref{AB2}, and get the following equation
             \begin{equation}
                s_n(\mathcal{P})=\sum_{k=1}^n(k+1)T(n-1,k)\label{AB3}
             \end{equation}
             Notice that the right-hand side of the equation is a hypergeometric series. The study of the hypergeometric identities \cite{Fine,Gasper,AeqB} has been going on for decades, and now ``the discoveries and the proofs of hypergeometric identities have been very largely automated'' \cite{AeqB}. We simplified this hypergeometric series by computer and got that
             \begin{equation}
                 \sum_{k=1}^n(k+1)T(n-1,k)=(2n)!!\label{AB4}
             \end{equation}
             By the equation \ref{AB3} and \ref{AB4}, the proof of Theorem \ref{theorem} is completed.
         \end{proof}
         Now, we enumerate the standard puzzle sequence of the union of $\{A_1, A_2, A_3\}$ or $\{A_1, A_2\}$ and  $B_i$ (resp. $C_i$).
         \begin{theorem}Let $\mathcal{P} = \{A_1, A_2, A_3\}$.
            \begin{align}
                &s_n(\mathcal{P}\cup B_1)=s_n(\mathcal{P}\cup B_2)=s_n(\mathcal{P}\cup B_3)=\frac43(2n+1)!!\label{1B123} \\
                &s_n(\mathcal{P}\cup B_4)=s_n(\mathcal{P}\cup B_5)=2^n(n+1)!\label{1B45}\\
                &s_n(\mathcal{P}\cup B_6)=(n+3)(2n+1)!!-(2n+2)!!\label{1B6}
             \end{align}
             \label{4.2}
         \end{theorem}
         \begin{proof}
             The puzzle with $\mathcal{P}\cup B_i$ as support, can be simply divided into two categories, whether it contains $B_i$ as its minimal support or not. If $B_i$ is not a member of its minimal support, this puzzle will degenerate into $S_n(\{A, B, D\})$, which has been shown in Section \ref{Se:3.3}. Then we just need to count the standard puzzle sequence with $\mathcal{P}\cup B_i$ as its minimal support.

             Firstly, for the equation \ref{1B123} we switch the order of the top and bottom of the last column of these three standard puzzles. Then we notice that these puzzles are changed into $S_n(\mathcal{P})$ ending with $A_1, A_2$ and $A_3$ respectively.

             Therefore we just need to show that the number of $\{A, B, D\}^n$ ending with $A, B$ and $D$ are equal and equal to one third of $\bigl|\{A, B, D\}^n\bigr|$. Simply switch the three digits in the upper right corner, and then the equation \ref{1B123} can be shown.

             Secondly we give a proof of the equation \ref{1B6}. To see clearly the relation between the numbers, we draw the skeleton of $S_n(\mathcal{P}\cup B_6)$. Then we find it only has two numbers less than the lower right number $2n-k$ of $S_{n-1}(\mathcal{P}_{2n-k})$ plus 2. Therefore, we have the following equation.
             \begin{equation}
                 s_n(\mathcal{P}\cup B_6)
                 =\sum_{k=1}^{n-1}\binom{2n-k+1}{2}T(n-1,k)+(2n+1)!!\label{TB6}
             \end{equation}

             \begin{figure}[h]
                \centering
                \includegraphics[scale=0.4]{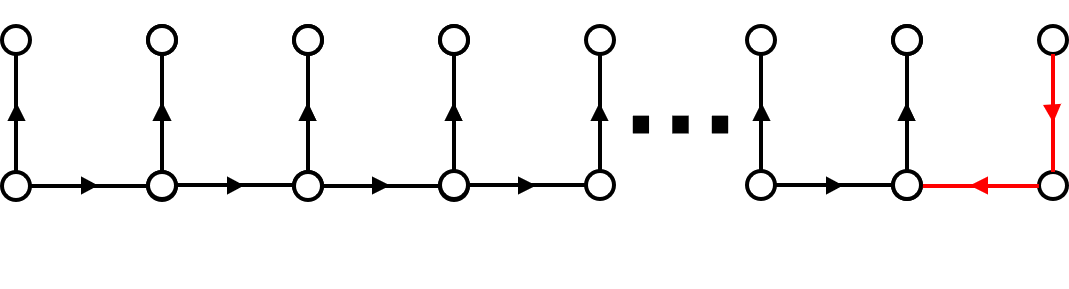}
                \caption{The skeleton of $\mathcal{P}\cup B_6$ with $B_6$}
                \label{fig:ABDQ}
             \end{figure}

             Lastly, for the equation \ref{1B45}, by swaping the upper left and the lower right of piece on the rightmost side of $S_n(\mathcal{P}\cup B_4)$, we can get a puzzle $S_n(\mathcal{P}\cup B_5)$. And vice versa. Then, the equivalence property has been shown. Additionally, notice that 
             \begin{align}
                 &s_n(\mathcal{P}\cup B_4\cup B_5)
                 -s_n(\mathcal{P})\notag\\
                 =&s_n(\mathcal{P}\cup B_4)
                 +s_n(\mathcal{P}\cup B_5)
                 -2s_n(\mathcal{P})\notag\\
                 =&2\bigl(s_n(\mathcal{P}\cup B_4)
                 -s_n(\mathcal{P})\bigr)
             \end{align}
             So we just need to count the numbers 
             $s_n(\mathcal{P}\cup B_4 \cup B_5)$. 
             We draw the skeleton to observe it.

             \begin{figure}[h]
                \centering
                \includegraphics[scale=0.4]{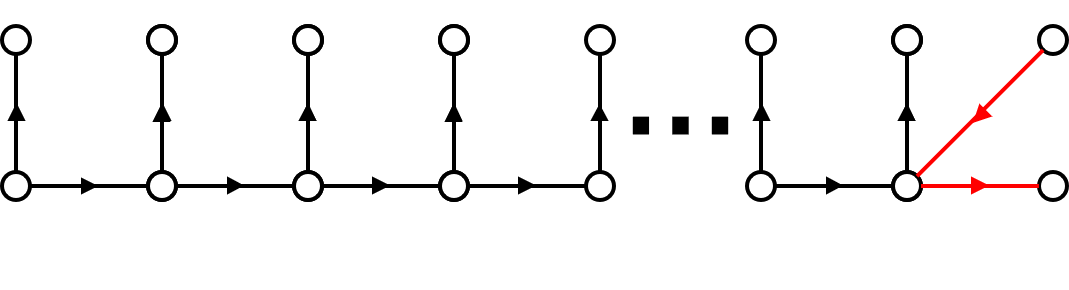}
                \caption{The skeleton of $S_n(\mathcal{P}\cup B_4 \cup B_5)$ with $B_4$ or $B_5$}
                \label{fig:ABDJL}
             \end{figure}

             According to this skeleton, we can write the following equation.
             \begin{equation}
                s_n(\mathcal{P}\cup B_4)
                =s_n(\mathcal{P}\cup B_5)
                =\frac{1}{2}\sum_{k=1}^{n-1}(2n-k)(k+1)T(n-1,k)+(2n+1)!!\label{TB45}
             \end{equation}
             In the end, we simplify the right-hand side of equations \ref{TB6} and \ref{TB45} by computer.
             \begin{align}
                 &\frac{1}{2}\sum_{k=1}^{n-1}(2n-k)(k+1)T(n-1,k)+(2n+1)!!=2^n(n+1)!\\
                 &\sum_{k=1}^{n-1}\binom{2n-k+1}{2}T(n-1,k)+(2n+1)!!=(n+3)(2n+1)!!-(2n+2)!!
             \end{align}
             Hence, the Theorem \ref{4.2} is proved. 
         \end{proof}
         Then, we give some conclusions about the standard puzzle sequence of $\{A_1, A_2\}\bigcup B_i$.
         \begin{theorem}Let $\mathcal{Q}=\{A_1, A_2\}$
             \begin{align}
                &s_n(\mathcal{Q}\cup B_1)
                =s_n(\mathcal{Q}\cup B_2)
                =s_n(\mathcal{Q}\cup B_3)
                =\frac32(2n)!! \\
                &s_n(\mathcal{Q}\cup B_4)
                =s_n(\mathcal{Q}\cup B_5)
                =(2n+2)(2n-1)!!-\frac12(2n)!!\\
                &s_n(\mathcal{Q}\cup B_6)
                =(2n^2+8n+1)(2n-2)!!-(4n+4)(2n-1)!!
             \end{align}\label{4.3}
         \end{theorem}
         \begin{proof}
             Similar to the proof of the previous theorem, we can get the following equation.
             \begin{align*}
                 &s_n(\mathcal{Q}\cup B_1)
                 =s_n(\mathcal{Q}\cup B_2)
                 =s_n(\mathcal{Q}\cup B_3)
                 =\sum_{k=1}^{n-1}\binom{k+3}{3}T(n-2,k)+(2n)!!\\
                 &s_n(\mathcal{Q}\cup B_4)
                 =s_n(\mathcal{Q}\cup B_5)
                 =\sum_{k=1}^{n-1}\binom{k+2}{2}(2n-k-1)T(n-2,k)+(2n)!!\\
                 &s_n(\mathcal{Q}\cup B_6)
                 =\sum_{k=1}^{n-1}\binom{2n-k}{2}(k+1)T(n-2,k)+(2n)!!
             \end{align*}
             And then by calculation, we can get equations as follows.
             \begin{align*}
                 &\sum_{k=1}^{n-1}\binom{k+3}{3}T(n-2,k)+(2n)!!
                 =\frac32(2n)!!\\
                 &\sum_{k=1}^{n-1}\binom{k+2}{2}(2n-k-1)T(n-2,k)+(2n)!!=(2n+2)(2n-1)!!-\frac12(2n)!!\\
                 &\sum_{k=1}^{n-1}\binom{2n-k}{2}(k+1)T(n-2,k)+(2n)!!=(2n^2+8n+1)(2n-2)!!-(4n+4)(2n-1)!!
             \end{align*}
             Hence, the proof of the Theorem \ref{4.3} is completed.
         \end{proof}
         Analogously, the standard puzzle sequence of the union of $\{A_1, A_2, A_3\}$ or $\{A_1, A_2\}$ and $\{C_i\}$ $(i\in\{1, 2, 3, 4, 5, 6\})$ can be enumerated as follows. And no proof will be given.
         \begin{theorem}
             Let $\mathcal{P}=\{A_1, A_2, A_3\}$ and $\mathcal{Q}=\{A_1, A_2\}$.
            \begin{align*}
                 &s_n(\mathcal{P}\cup C_1)
                 =s_n(\mathcal{P}\cup C_2)
                 =\binom{2n}{2}\cdot(2n-3)!!+(2n+1)!!\\
                 &s_n(\mathcal{P}\cup C_3)
                 =(2n+1)!!+(2n-1)!! \\
                 &s_n(\mathcal{P}\cup C_4)
                 =s_n(\mathcal{P}\cup C_5)
                 =s_n(\mathcal{P}\cup C_6)
                 =\binom{2n+1}{3}\cdot(2n-3)!!+(2n+1)!!\\
                 &s_n(\mathcal{Q}\cup C_1)
                 =\binom{2n-1}{2}\cdot(2n-4)!!+(2n)!!\\
                 &s_n(\mathcal{Q}\cup C_2)
                 =2\cdot(2n)!!-\binom{2n-1}{2}\cdot(2n-4)!!\\
                 &s_n(\mathcal{Q}\cup C_3)
                 =(2n)!!+(2n-2)!! \\
                 &s_n(\mathcal{Q}\cup C_4)
                 =\bigl[\binom{2n+1}{3}-1\bigr]\cdot(2n-4)!!+(2n)!!\\
                 &s_n(\mathcal{Q}\cup C_5)
                 =\frac{4n^2-7n+3}{3}(2n-4)!!+(2n)!!\\
                 &s_n(\mathcal{Q}\cup C_6)
                 =\binom{2n}{3}\cdot(2n-4)!!+(2n)!!\\
            \end{align*}
         \end{theorem}
     \subsection{Standard puzzles relating to Catalan numbers}
         Similar to the last subsection, in this subsection, 
         we enumerate the standard puzzle sequence of the union of $\{A_2, A_3\}$ or $\{A_2\}$ and  $B_i$ (resp. $C_i$) for $1\leq i\leq6$.

         \begin{theorem}
             Let $\mathcal{P}$ be $\{A_2, A_3\}$. Then we have
         \begin{equation}
             s_n(\mathcal{P}_{n+k+1})=t(n,k)\label{A23_P}
         \end{equation}
         where t(n,k) is the Catalan's triangle in OEIS \cite{OEIS} entry A009766, for $0 \leqslant k \leqslant n$.\label{k2}
    \end{theorem}
    \begin{proof}
         By definition of $t(n,k)$, we make the convention $t(n,n+1)=0$ and can get $t(n,k) = \sum_{j=0}^{k} t(n-1,j)$

         Moreover, for a puzzle in $S_{n}(\mathcal{P}_{n+k+1})$, the number to the left of $n+k+1$ can be $\{n, n+1, \dots, n+k\}$. Given the number to the left of the number $n+k+1$, we get rid of the rightmost column of this puzzle, and replace all the numbers by $\{1, 2, 3, \dots, 2n\}$ respecting the label ordering. Then we get a standard puzzle in
         $S_{n-1}(\mathcal{P}_{n+j})$ for $0 \leqslant j\leqslant n-1$. Besides, each element of $S_{n-1}(\mathcal{P}_{n+j})$ can be produced by only one elements of $S_{n}(\mathcal{P}_{n+k+1})$.
         
         Then we have,
         \begin{equation*}
            s_n(\mathcal{P}_{n+k+1})
            =\sum_{j=0}^{k}s_{n-1}(\mathcal{P}_{n+j})
         \end{equation*}
         Then the recurrence relation is proved. By checking the equation \ref{A23_P} for $n=1$ and this theorem is shown.
     \end{proof}
     Then, by checking in OEIS, we can get the following formula.
     \begin{equation}
        t(n,k)=\frac{n-k+1}{n+1}\binom{n+k}{n}
     \end{equation}
     \begin{theorem} Let $\mathcal{P}$ be $\{A_2, A_3\}$. Then,
        \begin{align}
            &s_n(\mathcal{P}\cup B_1)
            =\frac3{n+3}\binom{2n+2}{n}\\
            &s_n(\mathcal{P}\cup B_2)
            =\frac{7n+2}{n^2+2n}\binom{2n}{n+1}\\
            &s_n(\mathcal{P}\cup B_3)
            =C(n)+C(n-1)\\
            &s_n(\mathcal{P}\cup B_4)
            =\binom{2n+1}{n}\\
            &s_n(\mathcal{P}\cup B_5)
            =\frac{8n(2n+1)}{(n+2)(n+3)}\binom{2n-1}{n}
            +\frac{1}{n+2}\binom{2n+2}{n+1} \label{46}\\
            &s_n(\mathcal{P}\cup B_6)
            =\frac{3\binom{2n-1}{n}\binom{2n+2}{3}}{(n+2)(n+3)}
            +\frac{\binom{2n+2}{n+1}}{n+2}
        \end{align}\label{thm:4.6.}
     \end{theorem}
     \begin{proof}
         The proof of this result is quite similar to the proof for the Theorem \ref{4.2}. So we just give a proof of \ref{46} to be an example. Furthermore, we have listed all the results as follows.
        \begin{align}
             &s_n(\mathcal{P}\cup B_1)
             =\sum_{k=1}^{n-1}\binom{n-k+3}{3}t(n-2,k-1)+C(n)\\
             &s_n(\mathcal{P}\cup B_2)
             =\sum_{k=1}^{n-1}\binom{n-k+2}{2}t(n-2,k-1)+C(n)\\
             &s_n(\mathcal{P}\cup B_3)
             =C(n)+C(n-1)\\
             &s_n(\mathcal{P}\cup B_4)
             =\sum_{k=1}^n(n+k-1)t(n-1,k-1)+C(n)\\
             &s_n(\mathcal{P}\cup B_5)
             =\sum_{k=1}^n(n+k-1)(n-k+2)t(n-1,k-1)-\binom{2n+1}{n}+2C(n)\\
             &s_n(\mathcal{P}\cup B_6)
             =\sum_{k=1}^n\binom{n+k}{2}t(n-1,k-1)+C(n)
         \end{align}
         For the equation \ref{46}, notice that $L$ and $J$ can not appear in the $S_n(\mathcal{P}\cup B_4\cup B_5)$ at the same time. So we get 
         \begin{equation}
             s_n(\mathcal{P}\cup B_4\cup B_5)
             =s_n(\mathcal{P}\cup B_4)
             +s_n(\mathcal{P}\cup B_5)
             -s_n(\mathcal{P})\label{39}
         \end{equation}
         As with the previous theorem, let's draw the skeleton of $S_n(\mathcal{P}\cup B_4\cup B_5)$ first.

        \begin{figure}[h]
           \centering
           \includegraphics[scale=0.4]{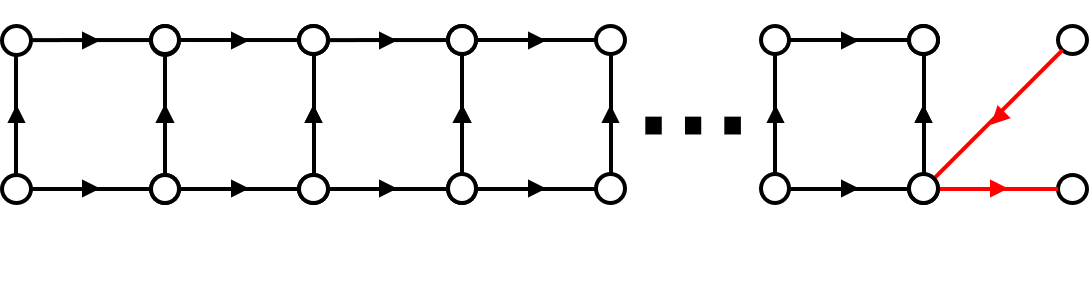}
           \caption{The skeleton of $S_n(\mathcal{P}\cup B_4\cup B_5)$ with $B_4$ or $B_5$}
           \label{fig:BDJL}
         \end{figure}

         From this picture, we can get
         \begin{equation}
            s_n(\mathcal{P}\cup B_4\cup B_5)
            =\sum_{k=1}^n(n+k-1)(n-k+2)t(n-1,k-1)+C(n)\label{48}
         \end{equation}
         Then, according to the equation \ref{39} and \ref{48}, we can get the following equation.
         \begin{equation}
             s_n(\mathcal{P}\cup B_5)
             =\sum_{k=1}^n(n+k-1)(n-k+2)t(n-1,k-1)
             -\binom{2n+1}{n}+2C(n)
         \end{equation}

         And then we can get the equations
     \end{proof}
     Furthermore, an argument similar to the Theorem \ref{4.2} and \ref{thm:4.6.} shows that 
     \begin{theorem} Let $\mathcal{Q}$ be $\{A_2\}$. Then,
        \begin{align*}
            &s_n(\mathcal{Q}\cup B_1)
            =\sum_{k=1}^{n-1}\binom{n-k+2}{2}t(n-2,k-1)+C(n-1)
            =\frac{1}{n+2}\binom{2n+2}{n+1}\\
            &s_n(\mathcal{Q}\cup B_2)
            =\sum_{k=1}^{n-1}(n-k+2)t(n-2,k-1)+C(n-1)
            =\frac{2}{n+1}\binom{2n}{n}\\
            &s_n(\mathcal{Q}\cup B_3)
            =C(n-1)+C(n-2)\\
            &s_n(\mathcal{Q}\cup B_4)
            =\sum_{k=1}^{n-1}(n+k-1)t(n-2,k-1)+C(n-1)
            =2\binom{2n-2}{n-1}\\
            &s_n(\mathcal{Q}\cup B_5)
            =\sum_{k=1}^{n-1}(n+k-1)(n-k+1)t(n-2,k-1)+C(n-1)
            =\frac{(2n^2+4)}{(n+1)(n+2)}\binom{2n}{n}\\
            &s_n(\mathcal{Q}\cup B_6)
            =\sum_{k=1}^{n-1}\binom{n+k}{2}t(n-2,k-1)+C(n-1)
            =\frac{n^2-n+4}{4n-2}\binom{2n+1}{n-1}
        \end{align*}
     \end{theorem}     
     \begin{remark}
         Let $\mathcal{P}=\{A_2, A_3\}$ and $\mathcal{Q}=\{A_2\}$. Notice that
         
         $\forall i \in \{1, 2, 3, 4, 5, 6\}$, $\exists j \in \{1, 2, 3, 4, 5, 6\}$ such that 
          \begin{align*}
              &F_1\circ F_2 \circ F_3(\mathcal{P}\cup C_i)=\mathcal{P}\cup B_j\\
              &F_1\circ F_2 \circ F_3(\mathcal{Q}\cup C_i)=\mathcal{Q}\cup B_j\\
          \end{align*}
         That is, $s_n(\mathcal{P}\cup C_i)=s_n(\mathcal{P}\cup B_j)$ and $s_n(\mathcal{Q}\cup C_i)=s_n(\mathcal{Q}\cup B_j)$. Therefore, we can enumerate $s_n(\mathcal{P}\cup C_i)$ and $s_n(\mathcal{Q}\cup C_i)$ by the sequence $s_n(\mathcal{P}\cup B_j)$ and $s_n(\mathcal{Q}\cup B_j)$ which we have enumerated before.
     \end{remark}
     \subsection{Standard puzzles relating to Entringer numbers}
     In 1966, Entringer \cite{Entringer} enumerated the down–up permutations according to the first element called Entringer numbers. Additionally, in the past two decades, a great deal of mathematical effort has been devoted to the study of Entringer numbers \cite{B-Entringer,MB-Entringer}.

     In this subsection, we enumerate a class of standard puzzle sequence which is related to Entringer numbers. That is, the standard puzzles with $\{A_1, A_2, A_3, A_4, A_5\}$ uniting a converter as a support. With the proposition \ref{pr}, the standard puzzle sequence with $\{A_1, A_2, A_4, A_5, A_6\}$ uniting a converter as a support is the same as this condition. Before we state and prove the theorem, we need a lemma as follows.

     \begin{lemma}
         Let $\mathcal{P}$ be $\{A_1, A_2, A_3, A_4, A_5\}$. Then we have
         \begin{align}
             s_n(\mathcal{P}_i)&=E(2n+1,2n+2-i)\\
             s_n(\mathcal{P}^j)&=(j-1)E(2n,j-2)
         \end{align}
         where E(n,k) is the number of down-up permutations of n+1 starting with k+1, in OEIS \cite{OEIS} entry A008282.
     \end{lemma}
     \begin{proof}
         By the definition of the Entringer numbers and the proof of Theorem \ref{thm:Secant}, It is straightforward to show this lemma. So the proof will be omitted.
     \end{proof}
     \begin{theorem}Let $\mathcal{P}$ be $\{A_1, A_2, A_3, A_4, A_5\}$. For $n\geq 2$, we have
         \begin{align}
             s_n(\mathcal{P}\cup B_1)&=
             s_n(\mathcal{P}\cup B_5)=
             s_n(\mathcal{P}\cup B_6)=\sum_{i=1}^{2n-2}\binom{i+3}{3}E(2n-2, i)+S(n+1)\\
             s_n(\mathcal{P}\cup B_2)&=s_n(\mathcal{P}\cup B_4)=\sum_{i=1}^{2n-2}(2n-1-i)\binom{i+2}{2}E(2n-2, i)+S(n+1)\\
             s_n(\mathcal{P}\cup B_3)&=\sum_{i=1}^{2n-2}(i+1)\binom{2n-i}{2}E(2n-2, i)+S(n+1)
         \end{align}\label{Thm:4.8.}
     \end{theorem}
     \begin{proof}
        The proof of this result is quite similar to the Theorem \ref{thm:4.6.} and so is omitted.
     \end{proof}
     \begin{remark}
         Let $\mathcal{P}=\{A_1, A_2, A_3, A_4, A_5\}$. Notice that
        
         $\forall i \in \{1, 2, 3, 4, 5, 6\}$, $\exists j \in \{1, 2, 3, 4, 5, 6\}$ such that 
         \begin{align*}
             &F_1\circ F_2 \circ F_3(\mathcal{P}\cup C_i)=\mathcal{P}\cup B_j\\
             &F_1\circ F_2 \circ F_3(\mathcal{Q}\cup C_i)=\mathcal{Q}\cup B_j\\
         \end{align*}
         That is, $s_n(\mathcal{P}\cup C_i)=s_n(\mathcal{P}\cup B_j)$. Therefore, we can enumerate $s_n(\mathcal{P}\cup C_i)$ by the sequence $s_n(\mathcal{P}\cup B_j)$ which we have enumerated in Theorem \ref{Thm:4.8.}.
     \end{remark}
     \subsection{A large class of Standard puzzles}
         \begin{theorem}
             Let $\mathcal{P}$, $\mathcal{Q}$ be the x-th and z-th simple pieces in the table \ref{appendix}(maybe the same simple piece), $\mathcal{T}$ be a standard piece $B_y$. Let $\overline{\mathcal{Q}}$ be the image of Q passing through the function $F_1 \circ F_2$ in definition \ref{Definition}.

             Then we have,
             \begin{equation}
                 s_n(\mathcal{P} \cup \mathcal{T} \cup\overline{\mathcal{Q}})
                 =\sum_{\substack{
                     i+j\leq 2m \\
                     k+l\leq 2p \\
                     m + p = n + 1}} 
                     P_x(i,j,m)\cdot T_y(i,j,k,l,m,p)\cdot P_z(k,l,p)+s_n(\mathcal{P})+s_n(\mathcal{Q})\label{eq:66}
             \end{equation}
             in which $i, j, k, l, m, p \in \mathbb{Z}_+ $,  $P_x(i, j, m)$ denotes the x-th standard puzzle of length m after given the rightmost two digits $i$ and $i + j$. The same thing with the definition of $P_z(k, l, p)$.\label{The Theorem}
         \end{theorem}
         To prove the theorem, we need a lemma.
         \begin{lemma}
             Let $N=\{1, 2, \dots, 2n+2\}$, and partition the set N into two parts A and B. Sort the elements in A and B from the smallest to the largest. Given the positive integers i, j, k, l, m, p, satisfying $i+j\leq 2m, k+l\leq 2p$ and $m+p=n+1$. $A=\{a_1, a_2, \dots, a_{2m}\}$ and $B=\{b_1, b_2, \dots, b_{2p}\}$.
             
             Then,

             1.The number of ways to partition the set N into A and B satisfying
             $a_i < b_k < b_{k+l} < a_{i+j}$ is $Q_1(i, j, k, l, m, p)$
             which is,
             \begin{equation}
                 Q_1(i,j,k,l,m,p)=
                 \sum_{\substack{
                 0\leq a\leq k-1\\
                 0\leq b\leq j-1}}
                 \binom{i+a-1}{a}
                 \binom{b+k+l-a-1}{b}
                 \binom{2m-i-b+2p-k-l}{2p-k-l}
             \end{equation}

             2.The number of ways to partition the set N into A and B satisfying
             $a_i < b_k < a_{i+j} < b_{k+l}$ is $Q_2(i, j, k, l, m, p)$
             which is,
             \begin{equation}
                 Q_2(i,j,k,l,m,p)=
                 \sum_{\substack{
                 0\leq a\leq k-1\\
                 0\leq b\leq j-1\\
                 0\leq c\leq l-1}}
                 \binom{i+a-1}{a}
                 \binom{b+k-a-1}{b}
                 \binom{c+j-b-1}{c}
                 \binom{2m+2p-k-c-i-j}{2m-i-j}
                 \label{Q}
             \end{equation}

             3.The number of ways to partition the set N into A and B satisfying
             $a_i < a_{i+j} < b_k < b_{k+l}$ is $Q_3(i, j, k, l, m, p)$
             which is,
             \begin{equation}
                Q_3(i,j,k,l,m,p)=
                 \sum_{a=0}^{k-1}
                 \binom{i+j+a-1}{a}
                 \binom{2m+2p-i-j-a}{2p-a}
             \end{equation}
         \end{lemma}
         \begin{proof}
             The proofs of these three formulas are similar, and we will only deal here with the proofs of equation \ref{Q} , which is as follows. 
             
             We consider the number $a_i$, $b_k$ and $a_{i+j}$. The number $a_i$ has $k$ choices $i, i+1, \dots, i+k-1$ and then $b_k$ has $j$ choices $i+k, i+k+1, \dots, i+j+k-1$ and $a_{i+j}$ has $l$ choices $i+j+k, i+j+k+1, \dots, i+j+k+l-1$. If each of these conditions is not true, the inequality will not be true. Let $a= a_i -i$, $b=b_k-i-k$, $c=a_{i+j}-i-j-k$, and then we enumerate the number of partitions meeting these conditions and get the equation.
             
         \end{proof}

             
         \paragraph{Proof of Theorem \ref{The Theorem}}
             \begin{proof}
                 Notice that $C_\alpha \notin \mathcal{P} \cup \mathcal{T} \cup\overline{\mathcal{Q}}$, for any $\alpha\in\{1, 2, 3, 4, 5, 6\}$. Then for a standard puzzle of $S_n(\mathcal{P} \cup \mathcal{T} \cup\overline{\mathcal{Q}})$ if the bottom number in one column is greater than the top number, then the bottom number in each subsequent column is greater than the top number. Therefore, the puzzle has three possible situations.

                 (1) The number at the top of each column is greater than the number at the bottom. Then its an element of $S_n(\mathcal{P})$.

                 (2) The numbers at the bottom of each column are greater than the numbers at the top. Then its an element of $S_n(\overline{\mathcal{Q}})$.

                 (3) There exists a number $m$ $(1\leq m\leq n)$. Then in the first m column(s), the top number is larger than the bottom number and the subsequent columns are the opposite. 

                 For the situation (1) and (2), it is a puzzle of $S_n(\mathcal{P})$ and $S_n(\overline{\mathcal{Q}})$. Then we just have to think about the situation (3). In this situation, we can cut this puzzle between the m-th column and the (m+1)-th column, and get two simple puzzle of $S_m(\mathcal{P})$ and $S_p(\mathcal{\overline{Q}})$($p=n+1-m$). Besides, for two simple puzzle P and $\overline{Q}$ of $S_m(\mathcal{P})$ and $S_p(\mathcal{\overline{Q}})$, let i and i+j be the number in the lower and top right corner of P, k and k+l be the number in the top and lower left corner of $\overline{Q}$, respectively.
                
                 we can get a partition of $\{1, 2, \dots, 2n+2\}$ into two parts A and B. $A=\{a_1, a_2, \dots, a_{2m}\}$ and $B=\{b_1, b_2, \dots, b_{2p}\}$ are sorted from the smallest to the largest. For given i, j, k, l, we let $a_i < b_k < b_{k+l} < a_{i+j}$ and replace the number $\alpha$ in P and $\beta$ in $\overline{Q}$ by $a_\alpha$ and $b_\beta$, respectively. Then we get a puzzle $\mathcal{P} \cup B_1 \cup\overline{\mathcal{Q}}$. Sum over all $i, j, k, l$ that satisfy $i+j\leq 2m$ and $k+l\leq2p$ and the equation \ref{eq:66} is set up for $i=1$ with $T_1(i, j, k, l, m, p)=Q_1(i, j, k, l, m, p)$. Analogously, for $y\in\{1, 2, 3\}$ the equation \ref{eq:66} holds and we have,
                 \begin{equation}
                     T_y(i, j, k, l, m, p)=Q_y(i, j, k, l, m, p)
                     \label{y123}
                 \end{equation} 
                 for $y\in\{1, 2, 3\}$. 
                 
                 For $\mathcal{T}=B_y$, $y\in\{4, 5, 6\}$, we use the proposition \ref{pr} about the map $F_1$ and $F_2$ in definition \ref{Definition} and get the equation 
                 \begin{equation}
                     s_n(\mathcal{P} \cup \mathcal{T} \cup\overline{\mathcal{Q}})=s_n(F_1 \circ F_2(\mathcal{P} \cup \mathcal{T} \cup\overline{\mathcal{Q}}))=s_n(\mathcal{\overline{P}} \cup \mathcal{\overline{T}} \cup\mathcal{Q})
                 \end{equation}
                 where $\mathcal{\overline{T}}=B_{y-3}$ for $y\in\{4, 5, 6\}$. We reuse the equations \ref{eq:66} and \ref{y123} for $y\in\{1, 2, 3\}$, the equation \ref{eq:66} is proved for all $i\in\{1, 2, 3, 4, 5, 6\}$. And for $i\in\{1, 2, 3\}$ we have,
                 \begin{equation}
                     T_{y+3}(i, j, k, l, m, p)=Q_y(k, l, i, j, p, m)
                 \end{equation}
                 Hence the proof is completed.
             \end{proof}
             Additionally, we list all the $P_x(i, j, m)$ $(1\leq x\leq 20, x\in \mathbb{Z})$ in the appendix (table \ref{appendix}). For $x\in\{1, 2, 3, 15, 16\}$, the conclusions are trivial and for $x\in\{5, 6, 11, 14\}$ they are corollaries of theorem \ref{theorem}. For $x\in \{17, 18, 19, 20\}$ and $x\in\{4, 7, 12, 13\}$ they are corollaries of theorem \ref{A123_} and theorem \ref{k2} respectively. When $x=8$ and $x=9$, we can get the function by the proof of theorem \ref{thm:Secant} and the definition of the Euler-Bernoulli or Entringer numbers in OEIS \cite{OEIS} entry A008282.

             Using the Theorem \ref{The Theorem} and the proposition \ref{pr}, we have the following corollary.
             \begin{corollary}
                 Let $\mathcal{P}$, $\mathcal{Q}$ be the x-th and z-th simple pieces in the table \ref{appendix}(maybe the same simple piece), $\mathcal{T}'$ be a standard piece $C_y$. Let $\mathcal{Q}'$ be the image of Q passing through the function $F_2$ in definition \ref{Definition}.

                 Then we have,
                 \begin{equation}
                     s_n(\mathcal{P}' \cup \mathcal{T}' \cup\mathcal{\overline{Q'}})
                     =\sum_{\substack{
                         i+j\leq 2m \\
                         k+l\leq 2p \\
                         m + p = n + 1}} 
                         P_x(i,j,m)\cdot T_y(i,j,k,l,m,p)\cdot P_z(k,l,p)+s_n(\mathcal{P})+s_n(\mathcal{Q})
                 \end{equation}
                 in which $i, j, k, l, m, p \in \mathbb{Z}_+ $,  $P_x(i, j, m)$ denotes the x-th standard puzzle of length m after given the rightmost two digits $i$ and $i + j$. The same thing with the definition of $P_z(k, l, p)$.\label{The Corollary}
                 \end{corollary}
 \section{Identities in standard puzzle sequences}\label{Se:5}
     \begin{theorem}
         Let $\mathcal{P}_i \in \{A_i, B_i\}$, $\mathcal{Q}_i \in \{C_i, D_i\}$ and $\alpha$ be a subset of \{1, 2, 3, 4, 5, 6\}. 
         
         So, we have
         \begin{equation}
            s_n\bigl(\bigl(\bigcup_{i\in\alpha}\mathcal{P}_i\bigr)\bigcup\bigl(\bigcup_{i\in\alpha}\mathcal{Q}_i\bigr)\bigr)=2s_n\bigl(\bigcup_{i\in\alpha} A_i\bigr)
         \end{equation}
         \label{Theorem 5.1}
     \end{theorem}
     \begin{proof}
         By the proposition \ref{pr} in Section \ref{Se:2}, we can know that 
         \begin{equation*}
             2s_n\bigl(\bigcup_{i\in\alpha} A_i\bigr)
             =s_n\bigl(\bigcup_{i\in\alpha} A_i\bigr)
             +s_n\bigl(\bigcup_{i\in\alpha} D_i\bigr)
         \end{equation*}
         Besides $A_i\in\{A_i, B_i\}$ and $D_i\in\{C_i, D_i\}$. So we only need to show that\\ $s_n\bigl(\bigl(\bigcup_{i\in\alpha}\mathcal{P}_i\bigr)\bigcup\bigl(\bigcup_{i\in\alpha}\mathcal{Q}_i\bigr)\bigr)$ are equal for any $\alpha\subset \{1, 2, 3, 4, 5, 6\}$, and any $\mathcal{P}_i \in \{A_i, B_i\}$, $\mathcal{Q}_i \in \{C_i, D_i\}$.

         Notice that the right-hand column of $\{A_i, B_i\}$(resp. $\{C_i, D_i\}$)is the same. 
         Then we check the reduction of each piece of a puzzle in $S_n\bigl(\bigl(\bigcup_{i\in\alpha}A_i\bigr)\bigcup\bigl(\bigcup_{i\in\alpha}D_i\bigr)\bigr)$ from right to left. If the piece does not belong to $\bigl(\bigcup_{i\in\alpha}\mathcal{P}_i\bigr)\bigcup\bigl(\bigcup_{i\in\alpha}\mathcal{Q}_i\bigr)$ as a standard puzzle, we turn the left-hand column of this pieces upside down and continue to retrieve until the research is complete. Because of this operation can swap $A_i$ with $B_i$, and swap $C_i$ with $D_i$, we can get a puzzle in $\bigl(\bigcup_{i\in\alpha}\mathcal{P}_i\bigr)\bigcup\bigl(\bigcup_{i\in\alpha}\mathcal{Q}_i\bigr)$. Besides, this mapping is easily seen to be invertible. Hence the proof is completed.


     \end{proof}
     \begin{remark}
         In fact, this idea is inspired by a proof in \cite{Han}.
     \end{remark}
     According to this theorem \ref{Theorem 5.1} and using the proposition \ref{pr} about the map $F_2$, we can get the following corollary easily.
     \begin{corollary}
        Let $\mathcal{P}_i \in \{A_i, C_i\}$, $\mathcal{Q}_i \in \{B_i, D_i\}$ and $\alpha$ be a subset of \{1, 2, 3, 4, 5, 6\}. 
        
        So, we have\label{corollary 5.1.}
        \begin{equation}
           s_n\bigl(\bigl(\bigcup_{i\in\alpha}\mathcal{P}_i\bigr)\bigcup\bigl(\bigcup_{i\in\alpha}\mathcal{Q}_i\bigr)\bigr)=2s_n\bigl(\bigcup_{i\in\alpha} A_i\bigr)
        \end{equation}
     \end{corollary}
     By the Theorem \ref{Theorem 5.1} and Corollary \ref{corollary 5.1.}, we can get many standard puzzle sequences from known sequences. For example, in Donald Knuth's paper \cite{Knuth}, he has shown the standard puzzle sequence of $\{A_1, A_4, B_3, B_6, C_3, C_6, D_1, D_4\}$, which is,
     \begin{theorem}
         \begin{equation}
             s_n\bigl(\{A_1, A_4, B_3, B_6, C_3, C_6, D_1, D_4\}\bigr)=W(n+1)
         \end{equation}
         where W(n) is the number of permutations p of $\{1\dots 2n\}$ such that $p[2k-1]<p[2k] \Longleftrightarrow  p[2k]<p[2k+1]$ in OEIS \cite{OEIS}, entry A261683.
     \end{theorem}
     Then, we can get the following corollary by Theorem \ref{Theorem 5.1}.
     \begin{corollary}
         Let $\mathcal{P}_i \in \{A_i, B_i\}$, $\mathcal{Q}_i \in \{C_i, D_i\}$ and $i\in\alpha = \{1, 3, 4, 6\}$. 
         
         So, we have
         \begin{equation}
            s_n\bigl(\bigl(\bigcup_{i\in\alpha}\mathcal{P}_i\bigr)\bigcup\bigl(\bigcup_{i\in\alpha}\mathcal{Q}_i\bigr)\bigr)=2s_n\bigl(\{A_1, A_3, A_4, A_6\}\bigr)=W(n+1)
         \end{equation}
         where W(n) is the number of permutations p of $\{1\dots 2n\}$ such that $p[2k-1]<p[2k] \Longleftrightarrow  p[2k]<p[2k+1]$ in OEIS \cite{OEIS}, entry A261683.
     \end{corollary}
     \begin{theorem}
         Let $\mathcal{P}$ be a subset of $\{A, B, C, D\}$, $\alpha$ be a subset of $\{1, 2, 3, 4, 5, 6\}$, and let $\mathcal{P}_\alpha $ be the set of each element of $\mathcal{P} $ with each element of $\alpha$ as a subscript. Then we have the following equation.
         \begin{equation}
             s_n\bigl(\mathcal{P}_\alpha\bigr)=
             s_n\bigl(A_\alpha\bigr)\cdot s_n\bigl(\mathcal{P}_1\bigr)
         \end{equation}
         \label{theorem:P=AP}
     \end{theorem}
     The proof of this result is quite similar to the theorem \ref{fibonacci} and so is omitted.

     Additionally, we have known all the standard puzzle sequence of $\mathcal{P}_1$. That is, if the cardinality of $\mathcal{P}_1$ is 1 or 2 or 4, the sequence is trivial. And for $\bigl|\mathcal{P}_1\bigr|=3$, we have
     \begin{align}
         &s_n(\{A_1, B_1, C_1\})=s_n(\{B_1, C_1, D_1\})=F(n+2)\label{eq:79}\\
         &s_n(\{A_1, B_1, D_1\})=s_n(\{A_1, C_1, D_1\})=n+2\label{eq:80}
     \end{align}
     With the help of the theorem \ref{fibonacci} and the proposition \ref{pr} about the map $F_2$ in the Section \ref{Se:2}, the equation \ref{eq:79} has been shown. And the equation \ref{eq:80} is trivial so the proof will be omitted.
 \section{Concluding remarks}
     In conclusion, the following enumerative problem have been solved except one kind of questions in case (2) and (3). 

     (1) The sequences of all 80 simple pieces.

     (2) The sequences of a 1-simple or 4-simple pieces with i-converters,for $i\in\{1, 2\}$.

     (3) The sequences of a 1-simple pieces uniting a 4-simple piece with i-converters for $i\in\{1, 2\}$.

     The unique unsolved case is the standard puzzle sequences of $\{A_1, A_2, A_4, A_5\}$ with converters (resp. with converters and another 4-simple piece). In addition, with the help of identities in standard puzzle sequences proposed in Section \ref{Se:5}, we can enumerate much more sequences. 
     
     Although we have enumerated tens of thousands of sequences, a lot of sequence are given in the form of hypergeometric series and we expect to simplify them.
     Moreover, the sequence of the union of several 1-simple pieces haven't been solved, some of which are conjectures proposed by Guo-Niu Han in \cite{Han} such as the sequences of $\{A_1, A_3\}$, $\{A_2, A_5\}$ and so on.

     Finally, as G.-N. Han said in \cite{Han}, standard puzzles is ``a large family of combinatorial objects'', ``defined by very simple rules''. After our research and previous studies, many classical sequence have been interpreted by standard puzzles, like double factorial numbers, secant numbers, Fibonacci numbers, Catalan numbers \cite{Han}, tangent numbers \cite{Han}, the numbers of $2 \times n$ whirlpool permutations \cite{Knuth} and so on. However, it is still tens of millions (about $2^{4!}=16,777,216)$ of sequence that haven't been enumerated. This is an issue for future research to explore.


 \subsection*{Acknowledgments}
 We are grateful to Prof. Shishuo Fu for many useful discussions and for introducing us to this problem.
 
 \subsection*{Appendix}
 \begin{table}
     \small
     \centering
     \begin{tabular}{c|c|c|c|c}
         \hline
         number & 1-basic skeletons & sequences & simple pieces & remark\\
         \hline \hline
         1 & \includegraphics[scale=0.25]{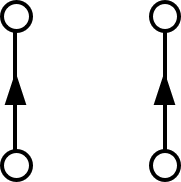}
         & $\binom{2n+2}{2, 2, \dots , 2}$ &  $\{A_1, A_2, A_3, A_4, A_5, A_6\}$ & trivial\\
         \hline
         2 & \includegraphics[scale=0.25]{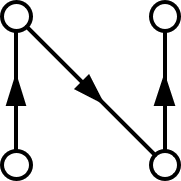}& 1 &  $\{A_3\}$ & trivial\\
         \hline
         3 & \includegraphics[scale=0.25]{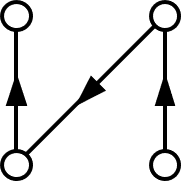}& 1 & $\{A_6\}$ & trivial\\
         \hline
         4 & \includegraphics[scale=0.25]{4.png}& $(2n+1)!!$ & $\{A_1, A_2, A_3\}$ & shown in \cite{Han}\\
         \hline
         5 & \includegraphics[scale=0.25]{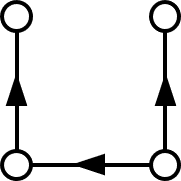}& $(2n+1)!!$ & $\{A_4, A_5, A_6\}$ & implied by entry 4\\ 
         \hline
         6 & \includegraphics[scale=0.25]{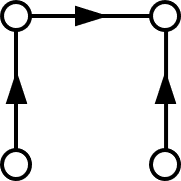}& $(2n+1)!!$ & $\{A_2, A_3, A_4\}$ & implied by entry 4\\ 
         \hline
         7 & \includegraphics[scale=0.25]{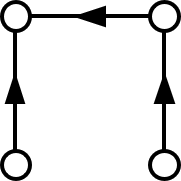}& $(2n+1)!!$ & $\{A_1, A_5, A_6\}$ & implied by entry 4\\
         \hline
         8 & \includegraphics[scale=0.25]{8.png}& secant numbers & $\{A_1, A_2, A_3, A_4, A_5\}$ & proved in Section 3\\
         \hline
         9 & \includegraphics[scale=0.25]{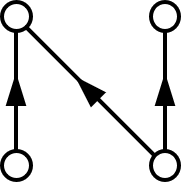}& secant numbers & $\{A_1, A_2, A_4, A_5, A_6\}$ & implied by entry 8\\
         \hline
         10 & \includegraphics[scale=0.25]{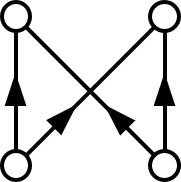}& \tabincell{c}{sequence in OEIS \\ entry A227656} & $\{A_1, A_2, A_4, A_5\}$ & proved in Section 3\\
         \hline
         11 & \includegraphics[scale=0.25]{12.png}& $(2n)!!$ & $\{A_1, A_2\}$ & shown in \cite{Han}\\
         \hline
         12 & \includegraphics[scale=0.25]{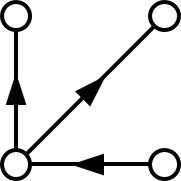}& $(2n)!!$ & $\{A_4, A_5\}$ & implied by entry 11\\
         \hline
         13 & \includegraphics[scale=0.25]{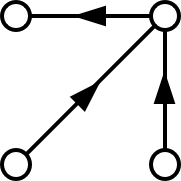}& $(2n)!!$ & $\{A_1, A_5\}$ & implied by entry 11\\
         \hline
         14 & \includegraphics[scale=0.25]{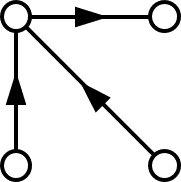}& $(2n)!!$ & $\{A_2, A_4\}$ & implied by entry 11\\
         \hline
         15 & \includegraphics[scale=0.25]{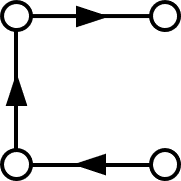}& 1 & $\{A_4\}$ & trivial\\
         \hline 
         16 & \includegraphics[scale=0.25]{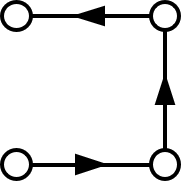}& 1 & $\{A_1\}$ & trivial\\
         \hline
         17 & \includegraphics[scale=0.25]{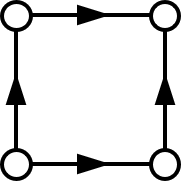}& $\frac{1}{n+2}\binom{2n+2}{n+1}$ & $\{A_2, A_3\}$ & shown in \cite{Han}\\ 
         \hline
         18 & \includegraphics[scale=0.25]{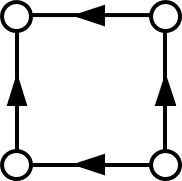}& $\frac{1}{n+2}\binom{2n+2}{n+1}$ & $\{A_5, A_6\}$ & implied by entry 13\\
         \hline
         19 & \includegraphics[scale=0.25]{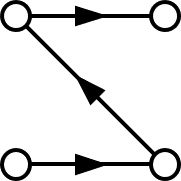}& $\frac{1}{n+1}\binom{2n}{n}$ & $\{A_2\}$ & shown in \cite{Han}\\
         \hline
         20 & \includegraphics[scale=0.25]{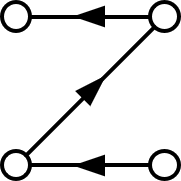}& $\frac{1}{n+1}\binom{2n}{n}$ & $\{A_5\}$ & implied by entry 19\\
         \hline
     \end{tabular}
     \caption{All the simple pieces and their sequences}
     \label{table}
 \end{table}

 \begin{table}
     \begin{enumerate}
         \item $\mathcal{P}=\{A_1, A_2, A_3, A_4, A_5, A_6\}$: $P_1(i, j, m)=\binom{2m-2}{2, 2, \dots , 2}$
         \item $\mathcal{P}=\{A_3\}$: $P_2(i, j, m)=1$ $(i=2m-1,j=1)$
         \item $\mathcal{P}=\{A_6\}$: $P_3(i, j, m)=1$ $(i=1,j=1)$
         \item $\mathcal{P}=\{A_1, A_2, A_3\}$: $P_4(i, j, m)=\frac{(i-1)!}{(2i-2m)!!} $ $(m\leq i\leq 2m-1)$
         \item $\mathcal{P}=\{ A_4, A_5, A_6\}$:      
             $P_5(i, j, m)=\left\{\begin{aligned}
                &(2m-3)!! &(i=1, m\neq 1)\\
             &1 &(i=1, m=1)\end{aligned}\right.$
         \item $\mathcal{P}=\{A_2, A_3, A_4\}$:      
             $P_6(i, j, m)=\left\{\begin{aligned}
                &(2m-3)!! &(&i+j=2m, m\neq 1)\\
             &1 &(&i=j=m=1)\end{aligned}\right.$
         \item $\mathcal{P}=\{A_1, A_5, A_6\}$: $P_7(i, j, m)=\frac{(2m-i-j)!}{(2m-2i-2j+2)!!} $ $(2\leq i+j\leq m+1)$
         \item $\mathcal{P}=\{A_1, A_2, A_3, A_4, A_5\}$: $P_8(i, j, m)=E(2m-2, i+j-2)$ where E(n, k) is the number of down-up permutations of n+1 starting with k+1 in OEIS \cite{OEIS} entry A008282.
         \item $\mathcal{P}=\{A_1, A_2, A_4, A_5, A_6\}$: $P_9(i, j, m)=E(2m-2, 2m-i)$ where E(n, k) is the number of down-up permutations of n+1 starting with k+1 in OEIS \cite{OEIS} entry A008282.
         \item $\mathcal{P}=\{A_1, A_2, A_4, A_5\}$: $P_{10}(i, j, m)$ hasn't been found.
         \item $\mathcal{P}=\{A_1, A_2\}$: 
         $P_{11}(i, j, m)=\left\{\begin{aligned}
            &\frac{(2m-1-i)\cdot(i-2)!}{(2i-2m)!!} &(&m\leq i \leq 2m-2, m\neq1)&\\&1 &(&i=j=m=1)
        \end{aligned}\right.$
         \item $\mathcal{P}=\{A_4, A_5\}$: 
         $P_{12}(i, j, m)=\left\{\begin{aligned}
             &(2m-4)!! &(&i=1,j\geq 3, m\neq 1)\\
             &1 &(&i=j=m=1)
         \end{aligned}\right.$
         \item $\mathcal{P}=\{A_1, A_5\}$: 
         $P_{13}(i, j, m)=\left\{\begin{aligned}
             &\frac{(i+j-2)\cdot(2m-1-j-i)!}{(2m+2-2j-2i)!!} &(&3\leq i+j \leq m+1,m\neq 1)
             \\&1 &(&i=1, j=2, m=1)
            \end{aligned}\right.$
         \item $\mathcal{P}=\{A_2, A_4\}$:
         $P_{14}(i, j, m)=\left\{\begin{aligned}
             &(2m-4)!! &(&i+j=2m, i\leq 2m-2, m\neq 1)\\
             &1 &(&i=j=m=1)
         \end{aligned}\right.$
         \item $\mathcal{P}=\{A_4\}$: $P_{15}(i, j, m)=1 $ $(i=1, j=2m-1)$
         \item $\mathcal{P}=\{A_1\}$: $P_{16}(i, j, m)=1 $ $(i=m, j=1)$
         \item $\mathcal{P}=\{A_2, A_3\}$: $P_{17}(i, j, m)=\frac{2m-i}{m}\binom{i-1}{m-1} $ $(m\leq i\leq 2m-1, i+j=2m)$
         \item $\mathcal{P}=\{A_5, A_6\}$: $P_{18}(i, j, m)=\frac{i+j-1}{m}\binom{2m-i-j}{m-1} $ $(i=1, 1\leq j\leq m)$
         \item $\mathcal{P}=\{A_2\}$: 
         $P_{19}(i, j, m)=\left\{\begin{aligned}
             &\frac{2m-i-1}{m-1}\cdot\binom{i-2}{m-2} &(&i+j=2m, m\leq i \leq 2m-2)\\&1&(&i=j=m=1)
         \end{aligned}\right.$
         \item $\mathcal{P}=\{A_5\}$: 
         $P_{20}(i, j, m)=\left\{\begin{aligned}
             &\frac{i+j-2}{m-1}\binom{2m-i-j-1}{m-2} &(&i=1,2\leq j \leq m)\\&1&(&i=j=m=1)
         \end{aligned}\right.$    
     \end{enumerate}
     \caption{All 20 simple pieces and their corresponding $P_x(i, j, m)$}
     \label{appendix}
\end{table}

 \end{document}